\documentclass
[
11pt]
{amsart}%

\usepackage{hyperref}
\usepackage{cite}
\usepackage{amssymb}
\usepackage{amsmath}
\usepackage{amsthm}
\usepackage{amsfonts}
\usepackage{graphicx}
\usepackage{bbm}
\usepackage{xcolor}
\usepackage[toc,page]{appendix}
\usepackage{subfigure}
\usepackage{mathtools}
\usepackage{bm} 
\usepackage[normalem]{ulem}
\usepackage{comment}

\newtheorem{theorem}{Theorem}[section]

\newtheorem{notation}[theorem]{Notation}
\newtheorem{proposition}[theorem]{Proposition}

\newtheorem{lemma}[theorem]{Lemma}
\newtheorem{remark}[theorem]{Remark}

\theoremstyle{definition}
\newtheorem{definition}[theorem]{Definition}

\numberwithin{equation}{section}

\newcommand{\masha}[1]
{{\color{red} Masha says: #1}}
\newcommand{\marco}[1]
{{\color{blue} Marco says: #1}}

\newcommand{\R}{\mathbb{R}}  
\newcommand{\E}{\mathbb{E}} 
\newcommand{\Prob}{\mathbb{P}}

\newcommand{\Hei}{\mathbb{H}}

\DeclareMathOperator{\Span}{span}

\begin{document}

\title[Small deviations, Chung's LiL on Heisenberg group]{Small deviations and Chung's law of iterated logarithm for  a hypoelliptic Brownian motion on the Heisenberg group}

\author[Carfagnini]{Marco Carfagnini{$^{\dag }$}}
\thanks{\footnotemark {$\dag$} Research was supported in part by NSF Grants DMS-1712427 and DMS-1954264.}
\address{Department of Mathematics\\
University of Connecticut\\
Storrs, CT 06269,  U.S.A.}
\email{marco.carfagnini@uconn.edu}

\author[Gordina]{Maria Gordina{$^{\dag }$}}
\address{ Department of Mathematics\\
University of Connecticut\\
Storrs, CT 06269,  U.S.A.}
\email{maria.gordina@uconn.edu}

\keywords{Diffusion processes, Heisenberg group, functional limit laws}

\subjclass{Primary 60F17; Secondary  60F15, 60G51, 60J65}


\begin{abstract}
A small ball problem and Chung's law of iterated logarithm for a hypoelliptic Brownian motion in Heisenberg group are proven. In addition, bounds on the limit in Chung's law are established.
\end{abstract}

\maketitle
\tableofcontents

\section{Introduction}

Let $E$ be a topological space  and $\left\{ X_{t} \right\}_{0 \leqslant t \leqslant 1}$ be an $E$-valued stochastic process with continuous paths such that $X_{0} = x_{0} \in E$ a.s. Denote by $W_{x_0}\left( E \right)$ the space of $E$-valued continuous functions on $[0, 1]$ starting at $x_0$, then we can view $X_t$ as a $W_{x_0}\left( E \right)$-valued random variable. Given a norm $\Vert \cdot \Vert$ on  $W_{x_0}\left( E \right)$, the small ball problem for $X_{t}$ consists in finding the rate of explosion of
\[
-\log \Prob \left( \Vert X  \Vert < \varepsilon \right)
\]
as $\varepsilon \rightarrow 0$. More precisely, a process $X_{t}$ is said to satisfy a \emph{small deviation principle} with rates $\alpha$ and $\beta$ if there exist a constant $c>0$ such that

\begin{equation}\label{e.general.s.d}
\lim_{\varepsilon \rightarrow 0} -\varepsilon^\alpha \vert \log \varepsilon \vert^\beta \log \Prob \left( \Vert X \Vert < \varepsilon \right) =c.
\end{equation}
The values of $\alpha, \beta$ and $c$ depend on the process $X_{t}$ and on the chosen norm on $W_{x_0} \left( E \right)$.
Small deviation principles have many applications including metric entropy estimates and Chung's law of the iterated logarithm. We refer to the survey paper  \cite{LiShao2001} for more details. In our paper we are mostly interested in connections of a small deviation principle to Chung's law of the iterated logarithm.

We say that a process $X_{t}$  satisfies \emph{Chung's law of the iterated logarithm} with rate $a \in \R_+$ if there exists a constant $C$ such that

\begin{equation}\label{e.general.chung.}
\liminf_{t\rightarrow \infty} \left(\frac{\log \log t}{t} \right)^a \max_{0\leqslant s \leqslant t} \vert X_{s} \vert = C.
\end{equation}
When $X_{t}$ is a Brownian motion, it was proven in a famous paper by K.-L.~Chung in 1948 that \eqref{e.general.chung.} holds with $a = \frac{1}{2}$ and $C= \frac{\pi}{\sqrt{8}}$.
If $W_{x_0} \left( E \right)$ is a Banach space, and the law $\mu$ of $X_{t}$ is a Gaussian measure on $W_{x_0} \left( E \right)$, then one can use a scaling property of the process $X_{t}$ to prove Chung's law of the iterated logarithm from a small deviation principle.

Small deviation principle for a Brownian motion and related processes have been extensively studied, we mention only a few most relevant to our results. In \cite{BaldiRoynette1992b} the authors considered  the case of a one-dimensional Brownian motion and H\"{o}lder norms, in \cite{KuelbsLi1993a} a Brownian sheet in H\"{o}lder norms has been considered, \cite{KhoshnevisanShi1998} studied the integrated Brownian motion in the uniform norm, and \cite{ChenLi2003}  the m-fold integrated Brownian motion in both the uniform and $L^2$-norm. In \cite{Remillard1994} a small deviation principle and Chung's law of iterated logarithm is proven for some stochastic integrals and in particular for L\'{e}vy's stochastic  area.

In the current paper we consider a hypoelliptic Brownian motion $g_{t}$ on the Heisenberg group $\Hei$ starting at the identity $e$ in $ \Hei$. The group $\Hei$ is the simplest example of a sub-Riemannian manifold, and it comes with a natural left-invariant distance, the Carnot-Carath\'{e}odory distance $d_{cc}$. We then consider the uniform norm
\[
\Vert g \Vert_{W_0 \left( \Hei \right) }:= \max_{0 \leqslant t \leqslant 1} \vert g_{t} \vert
\]
on the path space $W_0 \left( \Hei \right) $ of $\Hei$-valued continuous curves starting at the identity, where $ \vert \cdot \vert$ is a norm on $\Hei$ equivalent to the  Carnot-Carath\'{e}odory distance $d_{cc}$. We refer for details to Section \ref{sec.2}.

Our main results include Theorem \ref{t.chunglil} where we prove Chung's law of the iterated logarithm with $a= \frac{1}{2}$ for a hypoelliptic Brownian motion $g_{t}$. As a consequence of Theorem \ref{t.chunglil}, we prove Theorem \ref{t.smalldeviations} which represents a small deviation principle for the hypoelliptic diffusion $g_{t}$ with respect to the norm $\Vert \cdot \Vert_{W_0 \left( \Hei \right)}$. More precisely, we prove that there exists a finite positive constant $c$ such that \eqref{e.general.s.d} holds with $\alpha = 2$ and $\beta=0$, and we provide a lower and upper bound on $c$. Note that finding the constant $c$ explicitly is difficult even in more studied cases, see for example \cite[Remark 2.2]{KhoshnevisanShi1998}.

Let us explain now how our setting differs from known results. First observe that the hypoelliptic Brownian motion $g_{t}$ is an $\mathbb{R}^{3}$-valued stochastic process, but it is not a Gaussian process. Therefore we can not rely on the properties of Gaussian measures on Banach spaces, such as log-concavity and Anderson's inequality which are common tools in the subject. We refer to \cite{AndersonTW1955, Borell1976a} for more details about Gaussian measures on Banach spaces. These properties have been used to show the existence of a small deviations principle for some processes such as an integrated Brownian motion in \cite{KhoshnevisanShi1998}, and a Brownian motion with values in a finite dimensional Banach space in  \cite{DeAcosta1983a}.

Generally, if a small deviations principle is known, then it can be used together with scaling properties of the process to show Chung's law of the iterated logarithm. For example, in \cite{Remillard1994}  a small deviation principle for L\'{e}vy's stochastic area $A_{t}$ is first proven and then, using that $A_{\varepsilon t}  \stackrel{(d)}{=} \varepsilon A_{t}$ for any $t$ and $\varepsilon >0$, Chung's law of the iterated logarithm for the process $A_{t}$ follows. For related work we also refer to \cite{DobbsMelcher2014}. It is also possible to prove the converse. In \cite{KhoshnevisanShi1998} the authors first prove Chung's law of the iterated logarithm for the integrated one-dimensional Brownian motion $\int_0^t b_{s}ds$. Then, using the scaling property $\int_0^{\varepsilon t} b_{s} ds  \stackrel{(d)}{=} \varepsilon^{\frac{3}{2}} \int_0^t b_{s}ds$, a small deviation principle for $\int_0^t b_{s}ds$ is shown.

Most relevant to our work is \cite{KhoshnevisanShi1998}, where  the existence of the limit \eqref{e.general.s.d} for $X_{t} := \int_0^t b_{s}ds$  follows from Anderson's inequality for Gaussian measures. Chung's law of the iterated logarithm is then used to prove that the limit is finite. This method can not be used directly in our setting since the hypoelliptic Brownian motion $g_{t}$ is not a Gaussian process, and therefore we can not rely on Anderson's inequality. In our case we first prove in Proposition \ref{p.smalldeviations.estimates} that if the limit \eqref{e.general.s.d} exists then it is strictly positive and finite. We then prove Chung's law of the iterated logarithm  for $g_{t}$ and use it in place of Anderson's inequality to show the existence of the limit \eqref{e.general.s.d}. As a by-product we have bounds on this limit in terms of the lowest Dirichlet eigenvalues as given in Theorem \ref{t.ChungBounds}. The mathematical literature on the subject is vast, and we mention only the most relevant in terms of the techniques and results. In particular, a similar state space is considered in \cite{Neuenschwander2014a, LiuJ2013} though the results are different.

The paper is organized as follows. In Section \ref{sec.2} we describe the Heisenberg group $\Hei$ and the corresponding sub-Laplacian and hypoelliptic Brownian motion. In Section \ref{sec.3} we state the main results of this paper, namely Chung's law of the iterated logarithm in Theorem \ref{t.chunglil} and a small deviation principle in Theorem \ref{t.smalldeviations}. Section \ref{sec.4} contains  estimates that are needed to prove Theorem  \ref{t.chunglil} and Theorem \ref{t.smalldeviations}. We conclude \ref{sec.5} with the proof of the main results.

\section{Hypoelliptic Brownian motion on the Heisenberg group}\label{sec.2}

\subsection{Heisenberg group as Lie group}
The Heisenberg group $\Hei$ as a set is  $\R^3\cong \mathbb{R}^{2} \times \mathbb{R}$ with the group multiplication  given by

\begin{align*}
& \left( \mathbf{v}_{1}, z_{1} \right) \cdot \left( \mathbf{v}_{2}, z_{2} \right) := \left( x_{1}+x_{2}, y_{1}+y_{2}, z_{1}+z_{2} + \frac{1}{2}\omega\left( \mathbf{v}_{1}, \mathbf{v}_{2} \right)\right),
\\
& \text{ where } \mathbf{v}_{1}=\left( x_{1}, y_{1} \right), \mathbf{v}_{2}=\left( x_{2}, y_{2} \right) \in \mathbb{R}^{2},
\\
& \omega: \mathbb{R}^{2} \times \mathbb{R}^{2} \longrightarrow \mathbb{R},
\\
& \omega\left( \mathbf{v}_{1}, \mathbf{v}_{2} \right):= x_{1}y_{2}-x_{2} y_{1}
\end{align*}
is the standard symplectic form on $\mathbb{R}^{2}$. The identity in $\Hei$ is $e=(0, 0, 0)$ and the inverse is given by $\left( \mathbf{v}, z \right)^{-1}= (-\mathbf{v},-z)$.

The Lie algebra of $\Hei$ can be identified with the space  $\R^3\cong \mathbb{R}^{2} \times \mathbb{R}$  with the Lie bracket defined by
\[
\left[ \left( \mathbf{a}_{1}, c_{1} \right), \left( \mathbf{a}_{2}, c_{2} \right)  \right] = \left(0,\omega\left( \mathbf{a}_{1}, \mathbf{a}_{2} \right)  \right).
\]
The set $\R^3\cong \mathbb{R}^{2} \times \mathbb{R}$ with this Lie algebra structure will be denoted by $\mathfrak{h} $.

Let us now recall some basic notation for Lie groups. Suppose $G$ is a Lie group, then the left  and right multiplication by an element $k\in G$ are denoted by
\begin{align*}
L_{k}: G \longrightarrow G, &  & g \longmapsto k^{-1}g,
\\
R_{k}: G \longrightarrow G, &  & g \longmapsto gk.
\end{align*}

Recall that  the tangent space $T_{e}G$ can be identified with the Lie algebra $\mathfrak{g}$ of left-invariant vector fields on $G$, that is, vector fields $X$ on $G$  such that $dL_{k} \circ X=X \circ L_{k}$, where $dL_{k}$ is the differential of $L_k$. More precisely, if $A$ is a vector in $T_{e}G$, then we denote by $\tilde{A}\in \mathfrak{g}$ the (unique) left-invariant vector field such that $\tilde{A} (e) = A$.  A left-invariant vector field is determined by its value at the identity, namely, $\tilde{A}\left( k \right)=dL_{k} \circ\tilde{A}\left( e \right)$.

For the Heisenberg group the differential of left and right multiplication can be described explicitly as follows.
\begin{proposition}\label{p.Differentials}
Let $k= (k_1, k_2, k_3) = (\mathbf{k}, k_3 )$ and $g= (g_1, g_2, g_3) = (\mathbf{g}, g_3 )$ be two elements in $\Hei$. Then, for every $v= \left( v_1, v_2, v_3 \right) = (\mathbf{v}, v_3 )$ in  $T_g\Hei$, the differentials (pushforward) of the left and right multiplication are given by
\begin{align}\label{LeftRightMultDiff}
& dL_{k}=L_{k \ast}: T_g\Hei \longrightarrow T_{k^{-1}g}\Hei,  \notag
\\
& dR_{k}=R_{k \ast}: T_g\Hei \longrightarrow T_{gk}\Hei, \notag
\\
& dL_{k} (v) =  \left( v_1, v_2, v_3 + \frac{1}{2} \omega( \mathbf{v}, \mathbf{k}) \right), \notag \\
& dR_{k} (v) =  \left( v_1, v_2, v_3 + \frac{1}{2} \omega( \mathbf{v}, \mathbf{k}) \right).
\end{align}
\end{proposition}

\subsection{Heisenberg group as a sub-Riemannian manifold}
The Heisenberg group $\Hei$ is the simplest non-trivial example of a sub-Riemannian manifold.
We define $X$, $Y$ and $Z$ as the unique left-invariant vector fields satisfying $X_e = \partial_x$, $Y_e = \partial_y$ and $Z_e = \partial_z$ which are given by

\begin{align*}
& X = \partial_x - \frac{1}{2}y\partial_z,
\\
& Y = \partial_y + \frac{1}{2}x\partial_z,
\\
&  Z = \partial_z.
\end{align*}
Note that the only non-zero Lie bracket for these left-invariant vector fields is $[X, Y]=Z$, so the vector fields $\left\{ X, Y \right\}$ satisfy H\"{o}rmander's condition. We define the \emph{horizontal distribution} as $\mathcal{H}:= \Span \left\{  X, Y \right\}$ fiberwise, thus making $\mathcal{H}$ a sub-bundle in the tangent bundle $T\Hei$. To finish the description of the Heisenberg group as a sub-Riemannian manifold we need to equip the horizontal distribution  $\mathcal{H}$ with an inner product. For any $p \in \Hei$ we define the inner product $\langle \cdot , \cdot \rangle_{\mathcal{H}_{p}}$ on $\mathcal{H}_{p}$ so that $\left\{ X \left( p \right), Y \left( p \right) \right\}$ is an orthonormal (horizontal) frame at any $p \in \Hei$. Vectors in $\mathcal{H}_{p}$ will be called \emph{horizontal}, and the corresponding norm is denoted by $\Vert \cdot \Vert_{\mathcal{H}_{p}}$.

In addition, H\"{o}rmander's condition ensures that a natural sub-Laplacian on the Heisenberg group

\begin{equation}\label{e.2.1}
\Delta_{\mathcal{H}} =  X ^2 + Y ^2
\end{equation}
is a hypoelliptic operator by \cite{Hormander1967a}.

We recall now another notion in sub-Riemannian geometry, namely, of \emph{horizontal curves}. Suppose $\gamma(t) = \left( x\left( t \right), y\left( t \right), z\left( t \right) \right)=\left( \mathbf{x}\left( t \right), z\left( t \right) \right)$  is an absolutely continuous curve with values in $\Hei$, and  the corresponding tangent vector $\gamma^\prime(t)$ in $T\Hei_{\gamma(t)}$ is
\[
\gamma^\prime (t) = \left( x^\prime\left( t \right), y^\prime\left( t \right), z^\prime\left( t \right) \right)=\left( \bm{x}^\prime\left( t \right), z^\prime\left( t \right) \right).
\]
We denote by  $c_{g}$ the \emph{Maurer–Cartan form} on $\Hei$, i.e. the $\mathfrak{h}$-valued $1$-form on $\Hei$ defined by $c_{g}\left( v \right)=dL_{g} v$, $v \in T_{g}\Hei$. Note that the  pushforward of a vector in $T_{g}\Hei$ along the left translation can be found explicitly. Namely for $\gamma(t) = \left( \mathbf{x}\left( t \right), z\left( t \right) \right)$ the Maurer-Cartan form is
\begin{align}\label{e.MaurerCartan}
& c_{\gamma}\left( t \right)=c\left( t \right)=dL_{\gamma\left( t \right)}\left( \gamma^{\prime}(t)\right)
\\
& =\left( \mathbf{x}^\prime\left( t\right), z^{\prime}\left( t \right) -\frac{1}{2}\omega( \mathbf{x}\left( t \right), \mathbf{x}^{\prime}\left( t \right) )\right), \notag
\end{align}
where we used Proposition \ref{p.Differentials}.

\begin{definition}\label{Dfn.2.1}
An absolutely continuous curve $t \longmapsto \gamma (t) \in \Hei, t \in [0,1]$ is said to be \emph{horizontal} if $\gamma^{\prime}(t)\in\mathcal{H}_{\gamma(t)}$ for a.e. $t$, that is, the tangent vector to $\gamma\left(t\right)$ is horizontal for a.e. $t$. Equivalently we can say that $\gamma$ is horizontal if  $c_{\gamma}\left( t \right) \in \mathcal{H}_{e}$ for a.e. $t$.
\end{definition}
Equation \eqref{e.MaurerCartan} can be used  to characterize horizontal curves in terms of the components as follows. An absolutely continuous curve $\gamma$ is horizontal if and only if

\begin{equation}\label{e.horizontal}
z^{\prime}(t) -\frac{1}{2}\omega( \mathbf{x}\left( t \right), \mathbf{x}^{\prime}\left( t \right) ))=0 \text{ a.e. } t.
\end{equation}
The \emph{horizontal length} is defined as

\[
L_{\mathcal{H}}\left( \gamma \right):=\int_{0}^{1} \vert c_\gamma \left( s \right) \vert_{\mathcal{H}_e} ds,
\]
where we set $L_{\mathcal{H}}\left( \gamma \right)=\infty$ if $\gamma$ is not horizontal. The Heisenberg group as a sub-Riemannian manifold comes with a natural left-invariant distance.

\begin{definition}\label{Dfn.2.2}
For any $g_{1}, g_{2} \in \Hei$ the \emph{Carnot-Carath\'{e}odory distance} is defined as
\begin{align*}
& d_{cc} (g_{1}, g_{2}):= \inf \left\{  L\left( \gamma \right),  \gamma : [0,1] \longrightarrow \Hei, \gamma(0)=g_{1}, \gamma(1)=g_{2} \right\}.
\end{align*}
\end{definition}
Another consequence of  H\"{o}rmander's condition for left-invariant vector fields $X$, $Y$ and $Z$ is that by the Chow–Rashevskii theorem there exists a horizontal curve connecting any two points in $\Hei$, and therefore the Carnot-Carath\'{e}odory distance is finite on $\Hei$.

In addition to the Carnot-Carath\'{e}odory distance on the Heisenberg group,  we will use  the following homogeneous distance
\begin{equation}\label{hom.norm}
\rho( g_{1}, g_{2}) := \left(  \Vert \mathbf{x}_{1}-  \mathbf{x}_{2}\Vert^{4}_{\mathbb{R}^{2}}+ \vert z_{1}-z_{2} + \omega(\mathbf{x}_{1}, \mathbf{x}_{2}) \vert^2 \right)^{\frac{1}{4}},
\end{equation}
which is equivalent to the Carnot-Carath\'{e}odory distance, that is, there exist  two positive constants $c$ and $C$ such that
\begin{equation}\label{e.DistEquivalence}
c \rho( g_{1}, g_{2}) \leqslant d_{cc}( g_{1}, g_{2} ) \leqslant C \rho( g_{1}, g_{2})
\end{equation}
for all $g_{1}, g_{2} \in \Hei$. We denote by $\vert \cdot \vert$ the norm on $\Hei$ induced by $\rho$, that is, $\vert g \vert = \rho ( g, e )$ for all $g\in \Hei$. In particular, by the left-invariance of $\rho$ we have that for any $g_{1}, g_{2} \in \Hei$
\begin{equation}\label{e.triangular.ineq}
\vert g_{2}^{-1} g_{1} \vert = \rho\left( g_{2}^{-1} g_{1}, e \right) = \rho \left( g_{1}, g_{2} \right) \leqslant \rho \left( g_{1},e \right) +  \rho \left( g_{2},e \right) = \vert g_{1} \vert +\vert g_{2} \vert.
\end{equation}
This is discussed in a more general setting in \cite[Proposition 5.1.4]{BonfiglioliLanconelliUguzzoniBook}.

Finally, we need to describe  a hypoelliptic Brownian motion with values in $\Hei$. This is a stochastic process whose generator is the sub-Laplacian $\frac{1}{2}\Delta_{\mathcal{H}}$ defined by Equation \eqref{e.2.1}.

\begin{notation}\rm\label{ProbabilisticSetting}
Throughout the paper we use the following notation. Let $\left( \Omega, \mathcal{F}, \mathcal{F}_{t}, \mathbb{P}\right)$ be a filtered probability space. We denote the expectation under $\Prob$  by $\E$.

By a standard real-valued Brownian motion $\left\{ B_{t} \right\}_{t \geqslant 0}$ we mean a continuous adapted $\mathbb{R}$-valued stochastic process on $\left( \Omega, \mathcal{F}, \mathcal{F}_{t}, \mathbb{P}\right)$ such that for all $0 \leqslant s \leqslant t$ the increment $B_{t}-B_{s}$ is independent of $\mathcal{F}_{s}$ and has a normal distribution with mean $0$ and the variance $t-s$.

\end{notation}

\begin{definition}\label{d.HeisenbergBM} Let $W_{t}= \left( W_1(t), W_2(t), 0 \right)$ be an $\mathfrak{h}$-valued stochastic process, where $\bm{W}_{t}:=\left( W_1(t), W_2(t)\right)$ is a standard two-dimensional Brownian motion. A hypoelliptic Brownian motion $g_{t} = \left( g_1(t), g_2(t), g_3(t)\right)$ on $\Hei$ is the continuous $\Hei$-valued process defined by

\begin{equation}\label{e.HypoBM}
g_{t}:=\left( \bm{W}_{t}, A_{t}\right),
\end{equation}
where $A_{t} := \frac{1}{2} \int_0^t \omega\left( \bm{W}_{s}, d\bm{W}_{s}\right)$ is L\'{e}vy's stochastic area.
\end{definition}
Note that we used It\^{o}'s integral in this definition rather than Stratonovich' integral. However, these two integrals are equal in our setting since the symplectic form $\omega$ is skew-symmetric, and therefore  L\'{e}vy's stochastic area functional is the same for both integrals as was observed in \cite[Remark 4.3]{DriverGordina2008}.

One can also write a stochastic differential equation for $g_{t}=\left( x_{t}, y_{t}, z_{t} \right)$, $g_{0}=\left(0, 0, 0\right)=e \in \Hei$ as a stochastic differential equation for a Lie group-valued Brownian motion

\begin{align}\label{e.SDE}
& L_{g_{t} \ast}\left( dg_{t} \right)=g_{t}^{-1}dg_{t} =dW_{t},
\\
& g_{0}=e. \notag
\end{align}
Equation \eqref{e.HypoBM} gives an explicit solution to this stochastic differential equation for the Heisenberg group.

\section{Main results}\label{sec.3}

\begin{notation}
Let $X_{t}$ be a stochastic process with values in a metric space $\left( \mathfrak{X}, d \right)$ with $X_{0}=x \in \mathfrak{X}$, then  $X^{\ast}_{t}$ denotes the process defined by

\begin{align*}
X^{\ast}_{t}:= \max_{0\leqslant s \leqslant t} d\left( X_{s}, X_{0} \right).
\end{align*}
\end{notation}
For $\mathfrak{X}=\Hei$  we use the homogeneous distance $\rho$ with $X_{0}=e$, and on $\mathfrak{X}=\R^n$ we consider the standard Euclidean norm. Before formulating Chung's law of iterated logarithm for the hypoelliptic Brownian motion $g_{t}$ we introduce the notation

\begin{align*}
\phi\left( t \right):= \sqrt{ \frac{\log \log t}{t}}.
\end{align*}

\begin{theorem}[Chung's law of iterated logarithm]\label{t.chunglil}
Let $g_{t}$ be the hypoelliptic Brownian motion on the Heisenberg group $\Hei$ defined by \eqref{e.HypoBM}. Then there exists a constant $c \in ( 0, \infty )$ such that
\begin{equation}\label{e.chunglil}
\liminf_{t \rightarrow \infty} \phi\left( t \right) g_{t}^{\ast}=c \hskip0.1in \text{ a.s. }
\end{equation}
\end{theorem}

\begin{remark}\label{r.scaling}
Note that the hypoelliptic Brownian motion $g_{t}$ has the same scaling property with respect to the norm induced by the homogeneous norm $\rho$ as a standard Brownian motion in a Euclidean space. Indeed,

\begin{align*}
& \vert g_{\varepsilon t} \vert := \rho \left( g_{t \varepsilon}, e \right)= \sqrt[4]{ \vert B_{t \varepsilon } \vert^4 + A_{t \varepsilon}^2 }
\\
& \stackrel{(d)}{=} \sqrt[4]{ \vert \sqrt{\varepsilon} B_{t } \vert^4 +  \left( \varepsilon A_{t } \right) ^2 } = \sqrt{\varepsilon} \rho \left( g_{t}, e \right) = \sqrt{\varepsilon} \vert g_{t} \vert.
\end{align*}
Therefore it is not surprising that the process $g_{t}$ and the standard Brownian motion have the same rate $\phi (t)$ in  Chung's law of iterated logarithm.
\end{remark}
As a consequence of Theorem \ref{t.chunglil} we can prove a small deviation principle for $g_{t}$.

\begin{theorem}[Small deviation principle]\label{t.smalldeviations}
The limit
\begin{equation}\label{e.smalldeviations}
\lim_{\varepsilon \rightarrow 0 } -  \varepsilon^2 \log \Prob \left( g^{\ast}_{1} < \varepsilon \right)=c^{2}
\end{equation}
exists with constant $c$ being defined by \eqref{e.chunglil}.

\end{theorem}

\begin{remark}
Using scaling properties of $g_{t}$ we can formulate a small deviation principle over an interval $[0,T]$. Let $T>0$ be fixed, then by Theorem \ref{t.smalldeviations} and Remark \ref{r.scaling} we have that
\begin{align*}
& \lim_{\varepsilon \rightarrow 0} -\varepsilon^2 \log \Prob \left( g^{\ast}_T < \varepsilon \right) = c^2 T.
\end{align*}
Indeed,
\begin{align*}
& \lim_{\varepsilon \rightarrow 0} -\frac{\varepsilon^2}{T} \log \Prob \left( g^{\ast}_T < \varepsilon \right)  = \lim_{\varepsilon \rightarrow 0} -\frac{\varepsilon^2}{T} \log \Prob \left( g^{\ast}_1 < \frac{\varepsilon}{\sqrt{T}}\right)=c^2.
\end{align*}

\end{remark}

\begin{remark}
One might expect that the value of the limit, $c$, is the first Dirichlet eigenvalue for the hypoelliptic generator of the Brownian motion $g_{t}$ in the unit ball with respect to the homogeneous norm. Equivalently this constant can be expected to be described by the first exit time from this ball. This is a delicate issue since the infinitesimal generator is a hypoelliptic operator and the ball has a non-smooth boundary. We will address this problem in a forthcoming paper. Note that if the constant $c$ is indeed the first Dirichlet eigenvalue for the hypoelliptic operator in this set, then Theorem~\ref{t.ChungBounds} gives bounds for its value.
\end{remark}

\section{Preliminary estimates}\label{sec.4}

We collect here several preliminary estimates that will be used throughout the paper.
\begin{proposition}\label{p.general}
Let $Y_{t}$ be a positive real-valued process and assume  there exist two finite positive constants $0<a \leqslant b < \infty$ such that
\begin{align}
& \liminf_{t \rightarrow \infty } -  \frac{1}{t} \log \Prob \left( Y_{t} < 1 \right) \geqslant a
\label{e.4.1}
\\
& \limsup_{t \rightarrow \infty } -  \frac{1}{t} \log \Prob \left( Y_{t} < 1 \right) \leqslant b.
\label{e.4.2}
\end{align}
Let  $c$, $x$ and $y$ be real numbers such that $c>1$, and $0< x < a \leqslant b <y$,  then there exists an $n_0\in \mathbb{N}$ such that
\begin{equation}\label{e.generalestimate.below}
\sum_{n=n_0}^\infty \Prob \left( Y_{s_n} <1 \right) < \infty
\end{equation}
where $ s_n := \frac{1}{x} \log \log c^n$, and
\begin{equation}\label{e.generalestimate.above}
\sum_{n=n_0}^\infty \Prob \left( Y_{v_n} <1 \right) = \infty
\end{equation}
where $v_n$ is any positive sequence such that $v_n \rightarrow \infty $ as $n \rightarrow \infty$, and $v_n \leqslant \frac{1}{y} \log \log n^n $ for all $n \geqslant n_0$.
\end{proposition}
\begin{proof}
Let us first show \eqref{e.generalestimate.below}. By \eqref{e.4.1}  we have
\begin{align*}
& \Prob \left( Y_{t} <1 \right)  \leqslant e^{-at}
\end{align*}
for all large enough $t$. Therefore there exists an $n_1\in \mathbb{N}$ such that

\begin{align*}
& \sum_{n=n_1}^\infty \Prob \left( Y_{s_n} <1 \right) \leqslant  \sum_{n=n_1}^\infty e^{-a s_n}
\\
& = \sum_{n=n_1}^\infty e^{-\frac{a}{x} \log \log c^n} =  \sum_{n=n_1}^\infty \left( \frac{1}{n \log c} \right)^{\frac{a}{x}}
\end{align*}
which is a convergent series since $x<a$.

Let us now show \eqref{e.generalestimate.above}. By \eqref{e.4.1}  we have that

\begin{align*}
& \Prob \left( Y_{t} <1 \right)  \geqslant e^{-bt}
\end{align*}
for all large enough $t$, and hence there exists an $n_2 \in \mathbb{N}$ such that

\begin{align*}
&\sum_{n=n_2}^\infty \Prob \left( Y_{v_n} <1 \right)  \geqslant  \sum_{n=n_2}^\infty e^{-b v_n}
\\
& \geqslant \sum_{n=n_2}^\infty e^{-\frac{b}{y}  \log \log n^n} =\sum_{n=n_2}^\infty \left( \frac{1}{n \log n} \right)^{\frac{b}{y}}
\end{align*}
which is divergent since $b<y$. The proof is then completed by taking $n_0:= \max \left( n_1, n_2 \right)$.
\end{proof}

We first prove a weaker version of Theorem \ref{t.smalldeviations}, namely that if the limit in  \eqref{e.smalldeviations} exists, then it is finite and strictly positive. The estimates in Proposition \ref{p.smalldeviations.estimates} will be used in the proof of Chung's law of iterated logarithm. First we introduce the following notation.

\begin{notation}[Dirichlet eigenvalues in $\mathbb{R}^{n}$]\label{n.eigenvalues}
We denote by $\lambda^{(n)}_{1}$ the lowest Dirichlet eigenvalue of $-\frac{1}{2} \Delta_{\R^n}$ on the unit ball in $\mathbb{R}^{n}$, where $0< \lambda_{1}^{\left( n\right)} \leqslant \lambda_{2}^{\left( n\right)} \leqslant ...$  are Dirichlet eigenvalues for the Laplacian $-\frac{1}{2} \Delta_{\R^n}$ in the unit ball $D:= \left\{ x\in \R^n,  \vert x \vert <1 \right\}$.
\end{notation}

Recall that the lowest Dirichlet eigenvalues appear in a small deviation principle for a Brownian motion in $\mathbb{R}^{n}$, see e.g. \cite[Lemma 8.1]{IkedaWatanabe1989}. Namely, suppose $b_{t}$ is a standard Brownian motion in $\mathbb{R}^{n}$, then

\begin{equation}\label{e.s.d.brownian.mot}
\lim_{\varepsilon \rightarrow 0 } - \varepsilon^2 \log \Prob  \left( b^{\ast}_{1} < \varepsilon \right) = \lambda_1^{(n)},
\end{equation}
where $\lambda_1^{(n)}$ is as in Notation \ref{n.eigenvalues}, and
\[
b^{\ast}_1 := \max_{0\leqslant t \leqslant 1} \vert b_{t}\vert_{\R^n}.
\]

\begin{proposition}\label{p.smalldeviations.estimates}
Let $g_{t}$ be the hypoelliptic Brownian motion on the Heisenberg group $\Hei$. Set

\begin{align*}
& c_{-}:= \liminf_{\varepsilon \rightarrow 0 } -  \varepsilon^2 \log \Prob \left( g^{\ast}_{1} < \varepsilon \right),
\\
& c_{+}:= \limsup_{\varepsilon \rightarrow 0 }  - \varepsilon^2 \log \Prob \left( g^{\ast}_{1} < \varepsilon \right).
\end{align*}
Then
\begin{equation}\label{e.smalldeviations.estimates}
\lambda_{1}^{(2)} \leqslant c_{-}\leqslant c_{+} \leqslant c\left( \lambda_{1}^{(1)}, \lambda_{1}^{(2)} \right),
\end{equation}
where
\begin{align*}
& c\left( \lambda_{1}^{(1)}, \lambda_{1}^{(2)} \right):= f(x^{\ast})=\inf_{x \in \left( 0, 1 \right)} f\left( x \right),
\\
& f(x) = \frac{\lambda^{(2)}_1}{\sqrt{1-x}} + \frac{\lambda^{(1)}_1\sqrt{1-x} }{4x},
\\
& x^{\ast}= \frac{\sqrt{(\lambda_1^{(1)})^2 + 32\lambda^{(1)}_1\lambda^{(2)}_1 } -3\lambda^{(1)}_1} {2 \left( 4\lambda^{(2)}_1 - \lambda^{(1)}_1 \right)},
\end{align*}
and $\lambda^{(n)}_1$ are the lowest Dirichlet eigenvalues  on the unit ball as defined in Notation \ref{n.eigenvalues}.
\end{proposition}

\begin{proof}
The lower bound in \eqref{e.smalldeviations.estimates} follows from  the small deviation principle \ref{e.s.d.brownian.mot} for $\mathbb{R}^{n}$-valued Brownian motion and  the fact that $\Prob  \left(  g^{\ast}_{1} < \varepsilon \right) \leqslant \Prob  \left(  B^{\ast}_{1} < \varepsilon \right)$.

Let us prove now the upper bound. For any $x\in (0,1)$ we have
\begin{align*}
& \Prob  \left( g^{\ast}_{1} < \varepsilon \right) = \Prob  \left( \max_{ 0 \leqslant s \leqslant 1}  \left( \vert B_{s}\vert_{\R^2}^4 + \vert A_{s} \vert_{\R}^2 \right) < \varepsilon^4 \right)
\\
&\geqslant \Prob \left(   B^{\ast}_{1} < \left(1-x\right)^{\frac{1}{4}}\varepsilon , \;  A^{\ast}_{1} < \sqrt{x} \varepsilon^2 \right).
\end{align*}
It is well-known that $A_{t}=b_{\tau (t)}$ where $b_{t}$ is a one-dimensional Brownian motion independent of $B_{t}$, and $\tau(t) =\frac{1}{4} \int_0^t \vert B_{s} \vert^2_{\R^2}ds$, see for example \cite[Chapter 7, Section 6, Example 6.1]{IkedaWatanabe1989}. Therefore we have

\begin{align*}
&\Prob \left(  B^{\ast}_1 < \varepsilon \left(1-x\right)^{\frac{1}{4}}, \;  \sup_{ 0 \leqslant t \leqslant 1} \vert b_{\tau(t)} \vert_{\R} < \varepsilon^2\sqrt{x} \right)
\\
& = \Prob \left(  B^{\ast}_1 < \varepsilon\left(1-x\right)^{\frac{1}{4}}, \;  \sup_{ 0\leqslant t \leqslant \tau(1) } \vert b_{t}\vert_{\R} < \varepsilon^2\sqrt{x} \right)
\\
&\geqslant \Prob \left(    B^{\ast}_1< \varepsilon\left(1-x\right)^{\frac{1}{4}}, \;  \sup_{ 0\leqslant t \leqslant \frac{\varepsilon^2}{4} \left(1-x\right)^{\frac{1}{2}} } \vert b_{t}\vert_{\R} < \varepsilon^2 \sqrt{x}
\right)
\\
& = \Prob \left(     B^{\ast}_1< \varepsilon \left(1-x\right)^{\frac{1}{4}}, \;   b^{\ast}_1< \frac{2\varepsilon\sqrt{x}}{\left(1-x\right)^{\frac{1}{4}}}
\right)
\\
&= \Prob \left(  B^{\ast}_1 < \varepsilon\left(1-x\right)^{\frac{1}{4}} \right) \Prob \left(  b^{\ast}_1 < \frac{2\varepsilon\sqrt{x}}{\left(1-x\right)^{\frac{1}{4}}} \right).
\end{align*}
Thus
\begin{align*}
&\log  \Prob  \left( g^{\ast}_{1} < \varepsilon \right)  \geqslant \log \Prob \left(  B^{\ast}_1 < \varepsilon \left(1-x\right)^{\frac{1}{4}} \right) + \log \Prob \left(  b^{\ast}_1 < \frac{2\varepsilon \sqrt{x}}{\left(1-x\right)^{\frac{1}{4}}} \right),
\end{align*}
and hence
\begin{align*}
& -\varepsilon^2 \log \Prob  \left(  g^{\ast}_1 < \varepsilon \right)  \leqslant  -\varepsilon^2\left(1-x\right)^{\frac{1}{2}} \log \Prob \left(    B^{\ast}_1< \varepsilon \left(1-x\right)^{\frac{1}{4}} \right) \frac{1}{\left(1-x\right)^{\frac{1}{2}} }
\\
& - \varepsilon^ 2 \frac{4x}{\left(1-x\right)^{\frac{1}{2}}}  \log \Prob \left(  b^{\ast}_1 < \frac{2\sqrt{x}}{\left(1-x\right)^{\frac{1}{4}}}
\varepsilon \right)  \frac{ \left(1-x\right)^{\frac{1}{2}}}{ 4x}.
\end{align*}
From the small deviation principle  \eqref{e.s.d.brownian.mot} for a $\mathbb{R}^{n}$-valued  Brownian motion applied to $B_{t}$ and $b_{t}$ it follows that
\begin{align*}
& \limsup_{\varepsilon\rightarrow 0} - \varepsilon^2 \log \Prob  \left(  g^{\ast}_1< \varepsilon \right)  \leqslant  \frac{\lambda^{(2)}_1}{\sqrt{1-x}} + \frac{\lambda^{(1)}_1\sqrt{1-x} }{4x}
\end{align*}
for all $x$ in $(0, 1)$. Note that

\[
f\left( x \right):=\frac{\lambda^{(2)}_{1}}{\sqrt{1-x}} + \frac{\lambda^{(1)}_{1}\sqrt{1-x} }{4x} >0 \text{ for all } x\in \left( 0, 1 \right)
\]
always has a local minimum over $\left( 0, 1 \right)$ even if we do not rely on the known values of the eigenvalues $\lambda^{(2)}_1$ and $\lambda^{(1)}_{1}$. It is easy to see that this minimum is achieved at
\[
x^{\ast}= \frac{\sqrt{(\lambda_1^{(1)})^2 + 32\lambda^{(1)}_1\lambda^{(2)}_1 } -3\lambda^{(1)}_1} {2 \left( 4\lambda^{(2)}_1 - \lambda^{(1)}_1 \right)} \in \left( 0, 1 \right)
\]
which gives  Equation \eqref{e.smalldeviations.estimates}.
\end{proof}

\section{Proof of the main results}\label{sec.5}

\subsection{Chung's law of iterated logarithm for $g_{t}$}
The goal of this section is to prove Theorem \ref{t.chunglil}. Later in Proposition \ref{p.zero-one.law} we prove  that $c:= \liminf_{t\rightarrow \infty} \phi (t) g^{\ast}_{t}$ is constant a.s., where $\phi\left( t \right):= \sqrt{ \frac{\log \log t}{t}}$. For now $c$ is a random variable for which we first show lower and upper bounds in Proposition  \ref{p.chunglil.lowerbound} and Proposition \ref{p.chunglil.upperbound}.

\begin{proposition}[Lower bound] \label{p.chunglil.lowerbound}
For the lowest eigenvalue $\lambda_1^{(2)}$ as introduced in  Notation \ref{n.eigenvalues} we have

\[
c=\liminf_{t \rightarrow \infty} \phi\left( t \right) g_{t}^{\ast} \geqslant \sqrt{\lambda^{(2)}_1} \hskip0.1in \text{a.s.}
\]
\end{proposition}

\begin{proof}
While this proof is motivated by \cite{Remillard1994}, we provide  a detailed argument for completeness. Let $r>0$ be such that $0<r<\sqrt{\lambda_1^{(2)}}$. Then we can find a constant $M>1$ such that $rM<\sqrt{\lambda_1^{(2)}}$. We will show that
\[
\Prob \left( c<r \right) =  0 \text{ for all }   0<r<\sqrt{\lambda_1^{(2)}}.
\]
We have
\begin{align*}
&\Prob \left( c< r \right)= \Prob \left( \liminf_{t\rightarrow \infty} \phi (t) g^{\ast}_{t} < r \right) \leqslant \Prob \left( \bigcap_{k \geqslant 1 } \bigcup_{n \geqslant k}  \left\{ \inf_{M^n \leqslant t \leqslant  M^{n+2} } \phi (t) g^{\ast}_{t} < r \right\} \right)
\\
&\leqslant   \Prob \left( \bigcap_{k \geqslant 1 } \bigcup_{n \geqslant k}   \left\{ \frac{1}{M} \phi \left( M^n \right) g^{\ast}_{ M^n}  < r \right\} \right) = \Prob \left( \bigcap_{k \geqslant 1 } \bigcup_{n \geqslant k}   \left\{ g^{\ast} \left(  \frac{M^{n-2}}{r^2} \phi \left( M^n \right)^2\right)  < 1 \right\} \right),
\end{align*}
where $ g^{\ast} \left(  \frac{M^{n-2}}{r^2} \phi \left( M^n \right)^2\right)  := g^{\ast}_{ \frac{M^{n-2}}{r^2} \phi \left( M^n \right)^2}$. Here we used that
\begin{align*}
\inf_{a\leqslant t \leqslant b} \phi (t) g^{\ast}_{t} \geqslant \frac{\sqrt{a}}{ \sqrt{b}} \phi (a) g^{\ast}_a
\end{align*}
for any $0< a<b< \infty $. It is enough to show that

\begin{align*}
& \sum_{n=1} ^{\infty}\Prob \left( g^{\ast} \left(  \frac{M^{n-2}}{r^2} \phi \left( M^n \right)^2\right)  < 1  \right)  <\infty,
\end{align*}
and then the result follows from the Borel-Cantelli Lemma. By \eqref{e.smalldeviations.estimates} and the scaling property of $g_{t}$, it follows that
\begin{align*}
&  \lambda_1^{(2)} \leqslant  \liminf_{\varepsilon \rightarrow 0} -\varepsilon^2 \Prob \left( g^{\ast}_{1} < \varepsilon \right) =  \liminf_{\varepsilon \rightarrow 0}- \varepsilon^2 \Prob \left( \max_{ 0\leqslant s \leqslant 1} \vert g_{\frac{1}{\varepsilon^2} s} \vert <1 \right)
\\
&  = \liminf_{t \rightarrow \infty} -\frac{1}{t}\Prob \left( \max_{0\leqslant s \leqslant t} \vert g_{s} \vert <1 \right) =  \liminf_{t \rightarrow \infty} -\frac{1}{t}\Prob \left( g^{\ast}_{t} <1 \right).
\end{align*}
Moreover,
\begin{align*}
& \frac{M^{n-2}}{r^2} \phi \left( M^n \right)^2= \frac{1}{M^2 r^2}  \log \log M^n,
\end{align*}
and hence we can apply Proposition \ref{p.general} with $Y_{t}= g^{\ast}_{t}$, $a= \lambda_1^{(2)}$, $s_n= \frac{1}{x} \log \log M^n$,  and $x=M^2 r^2$, since $M^2r^2< \lambda_1^{(2)} $.
\end{proof}

Our next step is to show that $c$ is finite almost surely. To do so, we need the following Lemma.

\begin{lemma}\label{l.techincal}
Set $t_n=n^n$. Then for every $\varepsilon>0$
\begin{equation}\label{e.lemma}
\Prob \left( \bigcap_{k \geqslant 1 } \bigcup_{n \geqslant k}  \left\{ \phi\left( t_n \right) g^{\ast}_{t_{n-1} } > \varepsilon \right\} \right)=0
\end{equation}
\end{lemma}

\begin{proof}
It is enough to show that for any $\varepsilon >0$

\begin{align}
& \Prob \left( \bigcap_{k \geqslant 1 } \bigcup_{n \geqslant k}  \left\{ \phi\left( t_n \right) B^{\ast}_{ t_{n-1} } > \varepsilon \right\} \right)=0 \label{e1} \text{ and }
\\
&\Prob \left( \bigcap_{k \geqslant 1 } \bigcup_{n \geqslant k}  \left\{ \phi\left( t_n \right)^2 A^{\ast}_{ t_{n-1} } > \varepsilon \right\} \right)=0. \label{e2}
\end{align}

Indeed,
\begin{align*}
&  \left\{ \phi\left( t_n \right)g^{\ast}_{ t_{n-1} }> \varepsilon \right\} =  \left\{ \phi\left( t_n \right)^4 \max_{0\leqslant s \leqslant t_{n-1}} \left(  \vert B_{s} \vert_{\R^2}^4 + \vert A_{s} \vert^2  \right) > \varepsilon^4 \right\}
\\
& \subset  \left\{ B^{\ast}_{ t_{n-1} }  >\frac{ \varepsilon } {\sqrt[4]{2}\phi\left( t_n \right)} \right\}  \cup  \left\{  A^{\ast}_{t_{n-1} } >\frac{ \varepsilon^2 } {\sqrt{2}\phi\left( t_n \right)^2} \right\},
\end{align*}
and hence
\begin{align*}
& \Prob \left( \bigcap_{k \geqslant 1 } \bigcup_{n \geqslant k}  \left\{ \phi\left( t_n \right) g^{\ast}_{ t_{n-1} } > \varepsilon \right\} \right)
\\
& \leqslant \Prob \left( \bigcap_{k \geqslant 1 } \bigcup_{n \geqslant k}  \left\{ \phi\left( t_n \right) B^{\ast}_{t_{n-1} } >\frac{ \varepsilon} {\sqrt[4]{2}} \right\} \right)
\\
&  + \Prob \left( \bigcap_{k \geqslant 1 } \bigcup_{n \geqslant k}  \left\{ \phi\left( t_n \right)^2 A^{\ast}_{ t_{n-1} } >\frac{ \varepsilon^2} {\sqrt{2}} \right\} \right).
\end{align*}
Let us first prove \eqref{e1}. For every $\varepsilon$ fixed we have that
\begin{align*}
& \left\{ \omega \; : \;  \limsup_{n\rightarrow \infty } \phi \left(t_n \right)  \max_{0\leqslant s \leqslant t_{n-1}} \vert B_{s}(\omega) \vert_{\R^2} = 0 \right\}
\\
& \subset \bigcup_{k \geqslant 1 } \bigcap_{n \geqslant k} \left\{ \omega \; : \; \phi \left(t_n \right)  \max_{0\leqslant s \leqslant t_{n-1}} \vert B_{s}(\omega) \vert_{\R^2} < \varepsilon \right\} .
\end{align*}
Moreover, by \cite[Lemma 1]{JainPruitt1975} we have that for $t_n=n^n$
\begin{align*}
& \limsup_{n\rightarrow \infty } \phi \left(t_n \right)  \max_{0\leqslant s \leqslant t_{n-1}} \vert B_{s} \vert_{\R^2} = 0 \quad \text{a.s.}.
\end{align*}
Combining everything together we have that for $t_n=n^n$
\begin{align*}
&1= \Prob \left(  \limsup_{n\rightarrow \infty } \phi \left(t_n \right)  \max_{0\leqslant s \leqslant t_{n-1}} \vert B_{s} \vert_{\R^2} = 0 \right)
\\
& \leqslant \Prob \left( \bigcup_{k \geqslant 1 } \bigcap_{n \geqslant k}  \left\{ \phi \left(t_n \right)  \max_{0\leqslant s \leqslant t_{n-1}} \vert B_{s}(\omega) \vert_{\R^2} < \varepsilon \right\}  \right),
\end{align*}
and \eqref{e1} is proven since
\begin{align*}
& \Prob \left( \bigcap_{k \geqslant 1 } \bigcup_{n \geqslant k}  \left\{ \phi\left( t_n \right) B^{\ast}_{ t_{n-1} } > \varepsilon \right\} \right)
\\
& = 1- \Prob \left( \bigcup_{k \geqslant 1 } \bigcap_{n \geqslant k}  \left\{ \phi \left(t_n \right)  \max_{0\leqslant s \leqslant t_{n-1}} \vert B_{s}(\omega) \vert_{\R^2} < \varepsilon \right\}  \right).
\end{align*}
Let us now prove \eqref{e2}. It follows from \cite[Example on pp. 449-451]{Baldi1986a} that there exists a finite constant $d>0$ such that  with probability one we have $ A^{\ast}_{ t_{n-1} }  \leqslant d\, t_{n-1}^2 \phi \left( t_{n-1} \right)^2$ eventually, that is,
\[
\Prob \left( \bigcup_{k \geqslant 1} \bigcap_{n\geqslant k } \left\{ A^{\ast}_{ t_{n-1} }  \leqslant d\, t_{n-1}^2 \phi \left( t_{n-1} \right)^2 \right\} \right)=1.
\]
For any $\varepsilon>0$ there exists an $N_\varepsilon$ such that for any $n \geqslant N_\varepsilon$ we have that $ d\, t_{n-1}^2 \phi \left( t_{n-1} \right)^2 \phi \left( t_{n} \right)^2 \leqslant \varepsilon $. Set
\[
E_k :=\bigcap_{n\geqslant k }  \left\{ \phi \left( t_{n} \right)^2 A^{\ast}_{ t_{n-1} }  \leqslant d\, t_{n-1}^2 \phi \left( t_{n-1} \right)^2 \phi \left( t_{n} \right)^2 \right\}.
\]
Then the family $E_k$ is an  increasing sequence of sets, and it the follows that
\begin{align*}
& \bigcup_{k \geqslant 1} \bigcap_{n\geqslant k } \left\{ A^{\ast}_{ t_{n-1} }  \leqslant d\, t_{n-1}^2 \phi \left( t_{n-1} \right)^2 \right\} =  \bigcup_{k \geqslant 1}  E_k \subset \bigcup_{k \geqslant N_\varepsilon}  E_k
\\
& =  \bigcup_{k \geqslant N_\varepsilon}  \bigcap_{n\geqslant k }  \left\{ \phi \left( t_{n} \right)^2 A^{\ast}_{ t_{n-1} }  \leqslant d\, t_{n-1}^2 \phi \left( t_{n-1} \right)^2 \phi \left( t_{n} \right)^2 \right\}
\\
& \subset  \bigcup_{k \geqslant N_\varepsilon}  \bigcap_{n\geqslant k }  \left\{ \phi \left( t_{n} \right)^2 A^{\ast}_{ t_{n-1} }  \leqslant \varepsilon \right\} \subset  \bigcup_{k \geqslant 1}  \bigcap_{n\geqslant k }  \left\{ \phi \left( t_{n} \right)^2 A^{\ast}_{ t_{n-1} }  \leqslant \varepsilon \right\} .
\end{align*}
Therefore for any $ \varepsilon >0$  we have that
\[
\Prob \left(   \bigcup_{k \geqslant 1}  \bigcap_{n\geqslant k }  \left\{ \phi \left( t_{n} \right)^2 A^{\ast}_{ t_{n-1} }  \leqslant \varepsilon \right\} \right) =1,
\]
and so \eqref{e2} is proven.
\end{proof}

\begin{proposition}[Upper bound]\label{p.chunglil.upperbound}
Let $c\left( \lambda_{1}^{(1)}, \lambda_{1}^{(2)} \right)$ be as in Proposition \ref{p.smalldeviations.estimates}, then

\[
c=\liminf_{t \rightarrow \infty} \phi\left( t \right) g_{t}^{\ast}\leqslant \sqrt{c\left( \lambda_{1}^{(1)}, \lambda_{1}^{(2)} \right)} \hskip0.1in \text{ a.s. }
\]
\end{proposition}

\begin{proof}
Set $t_n=n^n$. We will show that $\Prob \left( c >r \right)=0$ for any $r> \sqrt{c\left( \lambda_{1}^{(1)}, \lambda_{1}^{(2)} \right)}$. Since
\begin{align*}
&   \Prob \left( \liminf_{t} \phi (t) g^{\ast}_{t} >r \right) \leqslant \Prob \left(   \bigcup_{k \geqslant 1 } \bigcap_{n \geqslant k}  \left\{  \phi (t_n) g^{\ast}_{t_n} >r   \right\} \right)
\\
& =  1- \Prob \left(  \bigcap_{k \geqslant 1 } \bigcup_{n \geqslant k}  \left\{  \phi (t_n) g^{\ast}_{t_n} \leqslant r   \right\} \right),
\end{align*}
 it is sufficient to show that for any $r> \sqrt{c\left( \lambda_{1}^{(1)}, \lambda_{1}^{(2)} \right)}$
\begin{equation}\label{e.idk.p}
\Prob \left( \bigcap_{k \geqslant 1 } \bigcup_{n \geqslant k}  \left\{ \phi\left( t_n \right) g^{\ast}_{ t_n } \leqslant r \right\} \right)=1.
\end{equation}

Fix $r > \sqrt{c\left( \lambda_{1}^{(1)}, \lambda_{1}^{(2)} \right)}$ and choose $r_1$ such that $ \sqrt{c\left( \lambda_{1}^{(1)}, \lambda_{1}^{(2)} \right)} <r_1<r$.  Let us define the events
\begin{align*}
& A_n :=  \left\{  \phi \left( t_n \right)\max_{t_{n-1} \leqslant s \leqslant t_n} \vert g_{t_{n-1}}^{-1} g_{s} \vert < r_1  \right\},
\\
& B_n := \left\{ \phi \left( t_n \right)  g^{\ast}_{ t_{n-1} }< \frac{r-r_1}{2} \right\}.
\end{align*}
Then by \eqref{e.triangular.ineq} on the event $A_n \cap B_n$ we have
\begin{align*}
& \phi \left( t_n \right) g^{\ast}_{ t_{n} }  \leqslant \phi \left( t_n \right) g^{\ast}_{ t_{n-1} }+ \phi \left( t_n \right) \max_{ t_{n-1} \leqslant s \leqslant t_n} \vert  g_{t_{n-1}}g_{t_{n-1}}^{-1}g_{s}  \vert
\\
& \leqslant  \phi \left( t_n \right) g^{\ast}_{ t_{n-1} }   +  \phi \left( t_n \right) \vert g_{t_{n-1}} \vert +  \phi \left( t_n \right) \max_{ t_{n-1} \leqslant s \leqslant t_n} \vert g_{t_{n-1}}^{-1}  g_{s}  \vert
\\
& \leqslant   \frac{r-r_1}{2}  +  \frac{r-r_1}{2}  + r_1 = r,
\end{align*}
and hence
\begin{align*}
&\Prob \left(  A_n \cap B_n  \right) \leqslant \Prob \left(  \phi \left( t_n \right) g^{\ast}_{ t_{n} }  \leqslant r \right).
\end{align*}
By Lemma \ref{l.techincal} we have $\Prob \left( \bigcap_{k \geqslant 1 } \bigcup_{n \geqslant k}  B_n^c \right) =0$, and therefore

\begin{align*}
&  \Prob \left( \bigcap_{k \geqslant 1 } \bigcup_{n \geqslant k}  \left\{ \phi \left( t_n \right)\max_{0 \leqslant s \leqslant t_n - t_{n-1}} \vert g_{t_{n-1}}^{-1} g_{s+ t_{n-1}} \vert < r_1  \right\} \right)
\\
& =  \Prob \left( \bigcap_{k \geqslant 1 } \bigcup_{n \geqslant k}  A_n \right) = \Prob \left( \bigcap_{k \geqslant 1 } \bigcup_{n \geqslant k}  \left(  A_n  \cap B_n \right) \right)
\\
&  \leqslant \Prob \left( \bigcap_{k \geqslant 1 } \bigcup_{n \geqslant k}  \left\{ \phi\left( t_n \right) g^{\ast} \left( t_n \right) \leqslant r \right\} \right).
\end{align*}
Equation \eqref{e.idk.p} holds if  we can show that

\begin{align*}
\Prob \left( \bigcap_{k \geqslant 1 } \bigcup_{n \geqslant k}  \left\{ \phi \left( t_n \right)\max_{0 \leqslant s \leqslant t_n - t_{n-1}} \vert g_{t_{n-1}}^{-1} g_{s+ t_{n-1}} \vert < r_1  \right\} \right)  =1.
\end{align*}
Note that  $ g_{t_{n-1}}^{-1}  g_{s+t_{n-1}} \stackrel{(d)}{=}  g_{s}$ and it is independent of $\mathcal{F}_{t_{n-1}} $. Indeed,
\begin{align}
 & g_{t_{n-1}}^{-1}  g_{s+t_{n-1}}\label{e.LeftIncrement}
 \\
 &   = \left( B_{s+t_{n-1}} - B_{t_{n-1}},  \frac{1}{2} \int_{t_{n-1}} ^{s+ t_{n-1}} \omega \left( B_u, dB_u \right)   + \frac{1}{2} \omega \left(  B_{s+t_{n-1}} , B_{t_{n-1}}  \right) \right)
 \notag
 \\
 & \stackrel{(d)}{=} \left( B_{s}, \frac{1}{2} \int_0^s \omega\left( B_u, dB_u \right) \right) = g_{s}.\notag
\end{align}
The assumptions of the Borel-Cantelli Lemma are satisfied, therefore we only need to show that the series with the term
\begin{align*}
&\Prob \left( \phi \left( t_n \right)\max_{0 \leqslant s \leqslant t_n - t_{n-1}} \vert g_{t_{n-1}}^{-1}  g_{s+ t_{n-1}} \vert < r_1  \right)
\end{align*}
diverges. We have
\begin{align*}
&\Prob \left( \phi \left( t_n \right)\max_{0 \leqslant s \leqslant t_n - t_{n-1}} \vert g_{t_{n-1}}^{-1} g_{s+ t_{n-1}} \vert < r_1  \right)
\\
&  =  \Prob \left( \max_{0 \leqslant s \leqslant t_n - t_{n-1}} \frac{\phi \left( t_n \right) }{r_1} \vert g_{s}\vert < 1  \right)
\\
&  = \Prob \left( g^{\ast} \left( \frac{t_n - t_{n-1}}{r_1^2} \phi\left( t_n \right)^2 \right) <1 \right),
\end{align*}
where $g^{\ast} \left( \frac{t_n - t_{n-1}}{r_1^2} \phi\left( t_n \right)^2 \right): = g^{\ast}_{ \frac{t_n - t_{n-1}}{r_1^2} \phi\left( t_n \right)^2 } $.

By \eqref{e.smalldeviations.estimates} and the scaling property of $g_{t}$ it follows that
\begin{align*}
&  c\left( \lambda_{1}^{(1)}, \lambda_{1}^{(2)} \right) \geqslant  \limsup_{\varepsilon \rightarrow 0} -\varepsilon^2 \Prob \left( g^{\ast}_{1} < \varepsilon \right) =  \limsup_{\varepsilon \rightarrow 0}- \varepsilon^2 \Prob \left( \max_{ 0\leqslant s \leqslant 1} \vert g_{\frac{1}{\varepsilon^2} s} \vert <1 \right)
\\
& = \limsup_{t \rightarrow \infty} -\frac{1}{t}\Prob \left( \max_{0\leqslant s \leqslant t} \vert g_{s} \vert <1 \right) =  \limsup_{t \rightarrow \infty} -\frac{1}{t}\Prob \left( g^{\ast}_{t} <1 \right).
\end{align*}
Moreover,
\begin{align*}
& \frac{t_n - t_{n-1}}{r_1^2} \phi\left( t_n \right)^2 = \frac{1}{r_1^2} \frac{t_n-t_{n-1}}{t_n} \log \log t_n
\\
& \leqslant \frac{1}{r_1^2} \log \log t_n = \frac{1}{r_1^2} \log \log n^n
\end{align*}
and hence we can apply Proposition \ref{p.general} with $Y_{t}= g^{\ast}_{t}$, $b= c\left( \lambda_{1}^{(1)}, \lambda_{1}^{(2)} \right)$, $v_n= \frac{1}{y} \log \log n^n$,  and $y=r_1^2$, since $r_1^2> c\left( \lambda_{1}^{(1)}, \lambda_{1}^{(2)} \right)$.
\end{proof}

\begin{remark}\label{r.not.right.multiplication}
In the proof of Proposition \ref{p.chunglil.upperbound} we used left increments $g_{t_{n-1}}^{-1} g_{s+t_{n-1}}$. It is easy to check that the argument does not work if one considers the right increments $g_{s+t_{n-1}}  g_{t_{n-1}}^{-1}$ instead. Indeed,
\begin{align*}
 & g_{s+t_{n-1}}  g_{t_{n-1}}^{-1}
 \\
 &   = \left( B_{s+t_{n-1}} - B_{t_{n-1}},  \frac{1}{2} \int_{t_{n-1}} ^{s+ t_{n-1}} \omega \left( B_u, dB_u \right)   -  \frac{1}{2} \omega \left(  B_{s+t_{n-1}} , B_{t_{n-1}}  \right) \right)
 \\
 & = \left( B_{s+t_{n-1}} - B_{t_{n-1}},  \frac{1}{2} \int_{0 } ^{s} \omega \left( B_{u+ t_{n-1}} - B_{t_{n-1}}, dB_{u + t_{n-1}} \right)  \right.
 \\
 & \left. +\frac{1}{2} \int_{0 } ^{s} \omega \left( B_{t_{n-1}} , dB_{u + t_{n-1}} \right)   -  \frac{1}{2} \omega \left(  B_{s+t_{n-1}} , B_{t_{n-1}}  \right) \right)
 \\
 &= \left( B_{s+t_{n-1}} - B_{t_{n-1}},  \frac{1}{2} \int_{0 } ^{s} \omega \left( B_{u+ t_{n-1}} - B_{t_{n-1}}, dB_{u + t_{n-1}} \right)  \right.
 \\
 &\left. +\frac{1}{2}  \omega \left( B_{t_{n-1}} , B_{s+ t_{n-1}} - B_{t_{n-1}} \right)   -  \frac{1}{2} \omega \left(  B_{s+t_{n-1}} , B_{t_{n-1}}  \right) \right)
 \\
  &= \left( B_{s+t_{n-1}} - B_{t_{n-1}},  \frac{1}{2} \int_{0 } ^{s} \omega \left( B_{u+ t_{n-1}} - B_{t_{n-1}}, dB_{u + t_{n-1}} \right)
  \right.
  \\
  &
  \left.+  \omega \left( B_{t_{n-1}} , B_{s+ t_{n-1}} - B_{t_{n-1}} \right)   \right)
 \\
 & \stackrel{(d)}{=} \left( B_{s}, \frac{1}{2} \int_0^s \omega\left( B_u, dB_u \right) + \omega \left( B_{t_{n-1}} , B_{s} \right) \right)
 \\
 & \not=  \left( B_{s}, \frac{1}{2} \int_0^s \omega\left( B_u, dB_u \right) \right)  = g_{s}.
\end{align*}
This is a consequence of our choice of the (left) Brownian motion $g_{t}$ defined by the left translation in \eqref{e.SDE}.
\end{remark}
The next statement completes the proof of Theorem \ref{t.chunglil}.

\begin{proposition}\label{p.zero-one.law}
Let $c$ be the random variable defined by \eqref{e.chunglil}, then $c$ is constant a.s.
\end{proposition}

\begin{proof}
Let $\mathcal{T}_u:= \sigma \left\{ B_r, \, r\geqslant u \right\}$, and $\mathcal{T}:= \cap_{u>0} \mathcal{T}_u$ be the tail $\sigma$-algebra generated by the Brownian motion, which is trivial by Kolmogorov's 0-1 law. We will show that
\begin{equation}\label{e.01law}
c^4= \liminf_{t\rightarrow \infty} \phi (t)^4 \max_{u \leqslant s \leqslant t }  \left[ \vert B_{s} \vert^4  + \left( \frac{1}{2} \int_u^s \omega \left(B_v, dB_v \right) \right)^2 \right].
\end{equation}
This means that the random variable $c$ is $\mathcal{T}_u$-measurable for every $u$ and hence $\mathcal{T}$-measurable. Since  $\mathcal{T}$ is trivial, then $c$ is constant a.s. because $\mathcal{T}$ is trivial. Let us now prove  \eqref{e.01law}. Suppose $u$ is fixed, and note that

\begin{equation}\label{e.idk}
c= \liminf_{t\rightarrow \infty} \phi (t) \max_{u\leqslant s \leqslant t} \vert g_{s} \vert.
\end{equation}
Indeed,
\begin{align*}
& \max_{u\leqslant s \leqslant t} \vert g_{s} \vert \leqslant  \max_{0\leqslant s \leqslant t} \vert g_{s} \vert \leqslant  \max_{u\leqslant s \leqslant t} \vert g_{s} \vert + \max_{0\leqslant s \leqslant u} \vert g_{s} \vert,
\end{align*}
and \eqref{e.idk} follows from the fact that $\lim_{t\rightarrow \infty }\phi(t)=0$. Using the triangular inequality one can show that
\begin{align}\label{e.idk2}
 &\max_{u \leqslant s \leqslant t }  \left[ \vert B_{s} \vert^4  + \left( \frac{1}{2} \int_u^s \omega \left(B_v, dB_v \right) \right)^2 \right] - A_u^2 -2\vert A_u \vert  \max_{0\leqslant s \leqslant t} \vert A_{s} \vert  \notag
\\
& \leqslant \max_{u\leqslant s \leqslant t} \vert g_{s} \vert^4 \\
& \leqslant  \max_{u \leqslant s \leqslant t }  \left[ \vert B_{s} \vert^4  + \left( \frac{1}{2} \int_u^s \omega \left(B_v, dB_v \right) \right)^2 \right]  +2A_u^2 + \vert  A_u \vert  \max_{0 \leqslant s \leqslant t }  \vert  A_{s} \vert \notag.
\end{align}
Indeed,
\begin{align*}
& \max_{u\leqslant s \leqslant t} \vert g_{s} \vert^4 = \max_{u \leqslant s \leqslant t }  \left[ \vert B_{s} \vert^4  +A_{s}^2 \right]
\\
& = \max_{u \leqslant s \leqslant t }  \left[ \vert B_{s} \vert^4   + \left( \frac{1}{2} \int_u^s \omega \left(B_v, dB_v \right)+A_u \right)^2\right]
\\
& = \max_{u \leqslant s \leqslant t }  \left[ \vert B_{s} \vert^4   + \left( \frac{1}{2} \int_u^s \omega \left(B_v, dB_v \right) \right)^2 + A_u^2 + A_u    \int_u^s \omega \left(B_v, dB_v \right) \right]  \\
& \leqslant \max_{u \leqslant s \leqslant t }  \left[ \vert B_{s} \vert^4   + \left( \frac{1}{2} \int_u^s \omega \left(B_v, dB_v \right) \right)^2 \right] + A_u^2 +\vert A_u \vert \max_{u \leqslant s \leqslant t }  \vert   \int_u^s \omega \left(B_v, dB_v \right) \vert
\\
&= \max_{u \leqslant s \leqslant t }  \left[ \vert B_{s} \vert^4   + \left( \frac{1}{2} \int_u^s \omega \left(B_v, dB_v \right) \right)^2 \right] + A_u^2 + \vert A_u \vert \max_{u \leqslant s \leqslant t }  \vert  A_{s} - A_u \vert  \\
& \leqslant \max_{u \leqslant s \leqslant t }  \left[ \vert B_{s} \vert^4   + \left( \frac{1}{2} \int_u^s \omega \left(B_v, dB_v \right) \right)^2 \right]  + 2A_u^2 + \vert  A_u \vert  \max_{0 \leqslant s \leqslant t }  \vert  A_{s} \vert,
\end{align*}
and the upper bound in \eqref{e.idk2} is proven. Let us now show the lower bound. We have that
\begin{align*}
& \max_{u \leqslant s \leqslant t }  \left[ \vert B_{s} \vert^4   + \left( \frac{1}{2} \int_u^s \omega \left(B_v, dB_v \right) \right)^2 \right] =  \max_{u \leqslant s \leqslant t }  \left[ \vert B_{s} \vert^4   + \left( A_{s}-A_u \right)^2 \right]
\\
& =  \max_{u \leqslant s \leqslant t }  \left[ \vert B_{s} \vert^4   + A_{s}^2 +A_u^2 -2A_uA_{s} \right] \leqslant \max_{u \leqslant s \leqslant t }  \vert g_{s} \vert^4 + A_u^2 + 2 \max_{u \leqslant s \leqslant t } \left( -A_u A_{s} \right) \\
& \leqslant \max_{u \leqslant s \leqslant t }  \vert g_{s} \vert^4 + A_u^2 + 2 \vert A_u \vert \max_{0\leqslant s \leqslant t } \vert A_{s} \vert,
\end{align*}
and the lower bound is also proven.

Before finishing the proof we recall that by \cite[Theorem 1]{Remillard1994}

\[
\liminf_{t\rightarrow \infty } \phi(t)^2  \max_{0\leqslant s \leqslant t} \vert A_{s} \vert = \frac{\pi}{4} \text{ a.s. }
\]
Then \eqref{e.01law} follows by \eqref{e.idk2} and the fact that $\lim_{t\rightarrow \infty }\phi(t)=0$.
\end{proof}

We have actually proven a more quantitative version of Theorem \ref{t.chunglil} as follows.

\begin{theorem}[Chung's law of iterated logarithm with bounds]\label{t.ChungBounds}
The constant $c= \liminf_{t\rightarrow \infty} \phi (t) g^{\ast}_{t}$ in Theorem \ref{t.chunglil} satisfies
\[
\sqrt{\lambda_1^{(2)}} \leqslant c \leqslant \sqrt{c\left( \lambda_{1}^{(1)}, \lambda_{1}^{(2)} \right)},
\]
where $\lambda_1^{(2)}$ is defined in Notation \ref{n.eigenvalues}, and $c\left( \lambda_{1}^{(1)}, \lambda_{1}^{(2)} \right)$ is defined in Proposition \ref{p.smalldeviations.estimates}.
\end{theorem}

\subsection{Small deviations for $g_{t}$ }
We are now ready to prove Theorem \ref{t.smalldeviations}, that is, the small deviation principle for the hypoelliptic Brownian motion $g_{t}$.

\begin{proof}[Proof  of Theorem \ref{t.smalldeviations}]
We recall the notation
\begin{align*}
& c_{-}:= \liminf_{\varepsilon \rightarrow 0 } -  \varepsilon^2 \log \Prob \left( g^{\ast}_{1} < \varepsilon \right),
\\
& c_{+}:= \limsup_{\varepsilon \rightarrow 0 }  - \varepsilon^2 \log \Prob \left( g^{\ast}_{1} < \varepsilon \right),
\end{align*}
and
\begin{align*}
&c:= \liminf_{t\rightarrow \infty} \phi (t) g^{\ast}_{t}.
\end{align*}
We first show that
\begin{equation}\label{e.firststep}
c_{+} \leqslant c^2 .
\end{equation}
Let $k\in \left( 0, c_{+} \right)$ be a fixed number, that is, $ k \leqslant \limsup_{\varepsilon \rightarrow 0} - \varepsilon^2 \log \Prob \left( g^{\ast}_{1} < \varepsilon \right)$. This means that there exists an $\varepsilon (k)$ such that
\begin{equation}\label{e.upperbound}
\Prob \left( g^{\ast}_{1} < \varepsilon \right) \leqslant  \exp \left( -\frac{k}{\varepsilon^2} \right),
\end{equation}
 for any $\varepsilon \leqslant \varepsilon (k)$.

Now fix $R>1$ and a number $\gamma $ such that $ 0 < R \gamma < k$. Define $t_n=R^n>1$, and $\varepsilon_n = \varepsilon_n(k, \gamma, R) := \frac{\sqrt{\gamma}}{\phi \left( t_{n+1} \right) \sqrt{t_n}}$. Note that $\varepsilon_n$ goes to zero as $n$ goes to infinity, and hence there exists an $N= N(k, \gamma, R )$ such that $\varepsilon_n \leqslant \varepsilon(k) $ for any $n \geqslant N(k, \gamma , R)$. Then we have that
\begin{align*}
&\Prob \left( g^{\ast}_{t_n} < \frac{\sqrt{\gamma}}{ \phi \left( t_{n+1} \right) } \right) = \Prob\left( g^{\ast}_{1} < \frac{\sqrt{\gamma} }{\phi \left( t_{n+1} \right)  \sqrt{t_n}} \right)  = \Prob\left( g^{\ast}_{1} < \varepsilon_n \right)
\\
& \leqslant   \exp \left( -\frac{k}{\varepsilon_n^2} \right) = \exp\left( -\frac{k}{\gamma}  \phi \left( t_{n+1} \right) ^2 t_n  \right) = \exp \left( -\frac{k}{\gamma} \frac{t_n}{t_{n+1}} \log \log t_{n+1}  \right)
\\
&  = \exp\left(-\frac{k}{R \gamma} \log \log R^{n+1} \right) =  \left( \frac{1}{ ( n+1) \log R} \right)^{\frac{k  }{R\gamma}},
\end{align*}
for all $n \geqslant N(k, \gamma , R)$, which is  a term of a convergent series since $R \gamma < k$.  Therefore for any $ 0 < k < c_{+}$, $R>1$, $0< R \gamma < k$, and $t_n=R^n$ we have that
\begin{align*}
& \sum_{n=1}^\infty \Prob \left( g^\ast_{t_n} < \frac{\sqrt{\gamma}}{ \phi \left( t_{n+1} \right) } \right)  < \infty.
\end{align*}
Hence by the Borel-Cantelli Lemma we have
\begin{align*}
& \Prob \left( \bigcap_{k \geqslant 1} \bigcup_{n \geqslant k}  \left\{ g^\ast_{t_n} < \frac{\sqrt{\gamma}}{ \phi \left( t_{n+1} \right) }  \right\} \right) = 0, \; \; \text{that is,}
\\
&\Prob \left( \bigcup_{k \geqslant 1} \bigcap_{n \geqslant k}  \left\{ g^\ast_{t_n} \geqslant \frac{\sqrt{\gamma}}{ \phi \left( t_{n+1} \right) }  \right\} \right) = 1.
\end{align*}
Therefore almost surely for all large $n$,  $g^\ast_{ t_n }  \geqslant \frac{ \sqrt{\gamma} } {\phi \left( t_{n+1} \right)}$. The function $\phi (t) $ is decreasing for $t>1$, therefore $t\in \left[ t_n , t_{n+1} \right]$  we have
\[
g^\ast_{t} \geqslant  g^\ast_{ t_n } \geqslant \frac{ \sqrt{ \gamma } } {\phi \left( t_{n+1} \right)} \geqslant  \frac{\sqrt{ \gamma }} {\phi \left( t \right)}
\]
which yields

\[
c:= \liminf_{t\rightarrow \infty} \phi (t) g^{\ast}_{t} \geqslant  \sqrt{ \gamma  } \quad \text{a.s.}
\]
for any $ \gamma < \frac{k}{R} <\frac{c_{+}}{R} $ , and hence  \eqref{e.firststep} is proven by letting first $R$ go to $1$, and then $k$ to $c_{+}$.

Let us now show that
\begin{equation}\label{e.secondstep}
c^2 \leqslant c_{-}.
\end{equation}
Suppose $ k> c_{-}$, then $k \geqslant \liminf_{\varepsilon \rightarrow0} - \varepsilon^2 \log \Prob \left( g^{\ast}_{1} < \varepsilon \right)$. Therefore there exists an $\varepsilon^\prime (k)$ such that
\begin{equation}\label{e.lowerbound}
\Prob \left( g^{\ast}_{1} < \varepsilon \right) \geqslant  \exp \left( -\frac{k}{\varepsilon^2} \right),
\end{equation}
for any $\varepsilon \leqslant \varepsilon^\prime (k)$. Now set $t_n =n^n$ and define
\begin{align*}
&\varepsilon_n= \varepsilon_n(k) := \frac{k}{\sqrt{c_{-}}} \frac{1}{ \sqrt{t_n - t_{n-1}} \phi \left( t_n \right)},
\\
&E_n^k := \left\{ \max_{t_{n-1} \leqslant s \leqslant t_n} \vert g_{t_{n-1}}^{-1} g_{s}  \vert< \frac{k}{\sqrt{c_{-}}} \frac{1}{ \phi \left( t_n \right)} \right\} .
\end{align*}
Note that $\varepsilon_n$ goes to zero as $n$ goes to infinity, and hence there exists an $N(k)$ such that $\varepsilon_n \leqslant \varepsilon^\prime (k)$ for any $ n\geqslant N(k)$. We claim that, for any $k>c_{-}$, $\sum_{n=1}^\infty \Prob \left( E^k_n \right) = \infty$. Indeed, since by \eqref{e.LeftIncrement} the left increment $g_{t_{n-1}}^{-1} g_{s+t_{n-1}}  \stackrel{(d)}{=} g_{t}$, we can use \eqref{e.lowerbound}  to see that  for $n \geqslant N(k)$
\begin{align*}
& \Prob\left( E_n^k \right) = \Prob \left( \max_{t_{n-1} \leqslant s \leqslant t_n} \vert g_{t_{n-1}}^{-1} g_{s}  \vert < \frac{k}{\sqrt{c_{-}}} \frac{1}{ \phi \left( t_n \right)} \right)
\\
&  =\Prob \left( \max_{0 \leqslant s \leqslant t_n - t_{n-1}}  \vert g_{t_{n-1}}^{-1} g_{s+t_{n-1}}  \vert <  \frac{k}{\sqrt{c_{-}}} \frac{1}{ \phi \left( t_n \right)} \right) = \Prob \left( \max_{0 \leqslant s \leqslant t_n - t_{n-1}}  \vert g_{s} \vert <  \frac{k}{\sqrt{c_{-}}} \frac{1}{ \phi \left( t_n \right)} \right)
\\
&  =\Prob \left( g^{\ast}_{1} < \varepsilon_n  \right) \geqslant  \exp \left( -\frac{k}{\varepsilon_n^2}\right) = \exp \left( - \frac{c_{-}}{k} \frac{t_n - t_{n-1}}{t_n}  \log \log t_n \right)
\\
&  \geqslant \exp \left( - \frac{c_{-}}{k}  \log \log t_n \right) = \left(  \frac{1}{n \log n} \right)^{ \frac{c_{-}}{k}}.
\end{align*}
This yields  $\sum_{n=1}^\infty \Prob \left( E^k_n \right) = \infty$ since $ k>c_{-}$. Note that the events $E_n$ are independent because the increments are independent and
 \[
 \max_{t_{n-1} \leqslant s \leqslant t_n} \vert g_{t_{n-1}}^{-1} g_{s}  \vert = \max_{0 \leqslant s \leqslant t_n - t_{n-1}}  \vert  g_{t_{n-1}}^{-1} g_{s+t_{n-1}} \vert
 \]
 and  $g_{t_{n-1}}^{-1} g_{s+t_{n-1}}$ is independent of $\mathcal{F}_{t_{n-1}}$ as shown in the proof of Proposition  \ref{p.chunglil.upperbound}. Hence by the Borel-Cantelli Lemma  we have that

 \[
 \Prob \left( \limsup_{n\rightarrow \infty} E_n^k \right)= \Prob \left( \bigcap_{j \geqslant 1 } \bigcup_{n \geqslant j}  \left\{ \phi\left(t_n \right) \max_{t_{n-1} \leqslant s \leqslant t_n} \vert g_{t_{n-1}}^{-1} g_{s}  \vert< \frac{k}{\sqrt{c_{-}}} \right\} \right)=1,
 \]
 which yields
 \[
 \liminf_{n\rightarrow \infty} \phi\left(t_n \right) \max_{t_{n-1} \leqslant s \leqslant t_n} \vert g_{t_{n-1}}^{-1}   g_{s} \vert< \frac{k}{\sqrt{c_{-}}}  \quad \text{ a.s. for all} \;\; k>c_{-},
 \]
and hence
\begin{equation}\label{e.almost.there}
 \liminf_{n\rightarrow \infty} \phi\left(t_n \right) \max_{t_{n-1} \leqslant s \leqslant t_n} \vert g_{t_{n-1}}^{-1} g_{s}  \vert \leqslant \sqrt{c_{-}} \quad \text{ a.s.}
\end{equation}
We will show later in the proof that
\begin{equation}\label{e.assumption}
\lim_{n\rightarrow \infty} \phi \left( t_n\right) g^\ast_{ t_{n-1} }= 0 \quad \text{a.s.}
\end{equation}
Assume \eqref{e.assumption} for now, we can use \eqref{e.triangular.ineq} to see that
\begin{align*}
& \phi \left(t_n \right) g^\ast_{t_n } \leqslant  \phi \left(t_n \right) g^\ast_{ t_{n-1} }  + \phi \left( t_n \right) \max_{t_{n-1} \leqslant s\leqslant t_n} \vert g_{s} \vert
 \\
& \leqslant  \phi \left(t_n \right) g^\ast_{ t_{n-1} }  +  \phi \left( t_n \right) \max_{t_{n-1} \leqslant s\leqslant t_n} \vert g_{t_{n-1}}^{-1} g_{s}  \vert + \phi \left( t_n \right) \vert g_{t_{n-1}} \vert
 \\
 & \leqslant  2  \phi \left(t_n \right) g^\ast_{ t_{n-1} } +  \phi \left( t_n \right) \max_{t_{n-1} \leqslant s\leqslant t_n} \vert g_{t_{n-1}}^{-1}  g_{s} \vert.
\end{align*}
Therefore by \eqref{e.almost.there} and \eqref{e.assumption} we have that

\begin{align*}
& c:=  \liminf_{t\rightarrow \infty} \phi (t) g^{\ast}_{t} \leqslant  \liminf_{n\rightarrow \infty} \phi \left( t_n \right) g^\ast_{ t_n }
\\
& \leqslant  \liminf_{n\rightarrow \infty}  \left( 2  \phi \left(t_n \right) g^\ast_{t_{n-1} }  +  \phi \left( t_n \right) \max_{t_{n-1} \leqslant s\leqslant t_n} \vert g_{s} g_{t_{n-1}}^{-1} \vert \right)  \leqslant \sqrt{c_{-}} \quad \text{ a.s.},
\end{align*}
which proves \eqref{e.secondstep}. Let us now show  \eqref{e.assumption}. By Lemma \ref{l.techincal} we have that for any $\varepsilon >0$
\begin{align*}
& 1= \Prob \left( \bigcup_{k \geqslant 1 } \bigcap_{n \geqslant k}  \left\{ \phi\left( t_n \right) g^\ast_{ t_{n-1} } < \varepsilon \right\}  \right) \leqslant \Prob \left(\limsup_{n\rightarrow \infty }  \phi\left( t_n \right) g^\ast_{ t_{n-1} }< \varepsilon  \right) .
\end{align*}
So for every $\varepsilon$ we have that
\[
\limsup_{n\rightarrow \infty }  \phi\left( t_n \right) g^\ast_{ t_{n-1} } < \varepsilon \quad \text{ a.s.,}
\]
and hence $\limsup_{n\rightarrow \infty }  \phi\left( t_n \right) g^\ast_{ t_{n-1} } =0$ a.s., which implies \eqref{e.assumption}.
\end{proof}

\providecommand{\bysame}{\leavevmode\hbox to3em{\hrulefill}\thinspace}
\providecommand{\MR}{\relax\ifhmode\unskip\space\fi MR }
\providecommand{\MRhref}[2]{%
  \href{http://www.ams.org/mathscinet-getitem?mr=#1}{#2}
}
\providecommand{\href}[2]{#2}

\end{document}